\documentclass[12pt,reqno]{amsart}
\usepackage{amsthm,amsmath,amssymb}
\usepackage{graphicx}
\usepackage{float}
\usepackage[colorlinks=true,
linkcolor=blue,
urlcolor=blue,
citecolor=blue]{hyperref}
\usepackage[T1]{fontenc}
\usepackage{mathtools}
\usepackage{relsize}
\usepackage{ytableau}
\usepackage{tikz}
\usepackage{enumerate}
\usepackage[toc,page]{appendix}
\usepackage{lscape}
\usepackage{multirow}
\usepackage{adjustbox,expl3,etoolbox} 
\usepackage[margin = 2.8cm]{geometry}
\usepackage{diagbox}
\usepackage{xcolor}
\usepackage{longtable}
\usepackage{comment}
\usepackage{amsaddr}
\usepackage{todonotes}
\usepackage{stmaryrd}
\usepackage{pdflscape}
\usepackage{tikz-cd}

\newtheorem{thm}{Theorem}[section]
\newtheorem{lem}[thm]{Lemma}
\newtheorem{prop}[thm]{Proposition}

\theoremstyle{definition}

\newtheorem{claim}{Claim}

\DeclareMathOperator\Aut{Aut}

\newcommand\sym[1]{\operatorname{Sym}(#1)}

\newcommand{\agl}[2]{\operatorname{AGL}_#1(#2)}
\newcommand{\pgl}[2]{\operatorname{PGL}_{#1}(#2)}

\newcommand{\pgammal}[2]{\operatorname{P\Gamma L}_#1(#2)}

\newcommand{\pg}[2]{\operatorname{PG}_{#1}(#2)}

\newcommand{\itbf}[1]{{\bf{{\emph{{#1}}}}}}

\begin{document}	
	
	\title[]{The automorphism group of the derangement graph of $\operatorname{PGL}_{2}(q)$ acting on the projective line}
	
	\author{Andriaherimanana Sarobidy Razafimahatratra}
	\address{Department of Mathematics and Computer Science, University of Lethbridge,\\
	Lethbridge, AB T1K 3M4, Canada}
	\email{\href{mailto:sarobidy@phystech.edu}{sarobidy@phystech.edu}}

	
	
	\date{September 25, 2026}
	
	\maketitle
	
	\begin{abstract}
		Given a finite transitive group $G\leq \operatorname{Sym}(\Omega)$, the derangement graph $\Gamma_G$ is the graph whose vertex set is $G$, and two vertices $g$ and $h$ are adjacent if the ratio $h^{-1}g$ is a fixed-point-free permutation. In this paper, we show that the automorphism group of the derangement graph of the transitive permutation group corresponding to the natural action of $\operatorname{PGL}_{2}(q)$ on the projective line $\operatorname{PG}_{1}(q)$ is
		\begin{align*}
			\operatorname{Aut}(\Gamma_{\operatorname{PGL}_{2}({q})}) = \left(L_{\operatorname{PGL}_2(q)}\times R_{\operatorname{PGL}_2(q)}\right) \rtimes \left(\langle \psi \rangle \times \gamma_{\operatorname{Aut}(\mathbb{F}_q)}\right),
		\end{align*} where $L_{\operatorname{PGL}_2(q)}$ is the left-regular representation of $\operatorname{PGL}_2(q)$, $R_{\operatorname{PGL}_2(q)}$ is the right-regular representation of $\operatorname{PGL}_2(q)$, $\gamma_{\operatorname{Aut}(\mathbb{F}_q)}$ is the group of conjugation by elements of $\operatorname{Aut}(\mathbb{F}_q)$, and $\psi: \operatorname{PGL}_2(q) \to \operatorname{PGL}_2(q)$ such that $\psi(x) = x^{-1}$. %As a by-product, we show that $\Gamma_{\operatorname{PGL}_2(q)}$ is not a normal Cayley graph.
	\end{abstract}
	
	\section{Introduction}
	
	Determining the full automorphism group of Cayley graphs is a central problem in algebraic graph theory. In general, the knowledge of the full automorphism group of a graph can help understand certain combinatorial structures associated with the graph. For instance, the articles \cite{Fang2002,Feng2006,Ganesan2013,Godsil1983,Imrich1976} are a few examples of the vast work in the literature where this problem is studied for  Cayley graphs. This paper is concerned with determining the full automorphism group of a Cayley graph associated with Erd\H{o}s-Ko-Rado type theorems in permutation groups.
	
	The Erd\H{o}s-Ko-Rado (EKR) theorem \cite{ERDOS1961} is an extremal set theory result on maximum intersecting families. Define $[n]: = \{1,2,\ldots,n\}$ for any integer $n\geq 1$ and let $1\leqslant k\leqslant n$. The EKR theorem asserts that if $n\geq 2k$ and $\mathcal{F}$ is a family of $k$-subsets of $[n]$ in which any two intersect (i.e., an intersecting family), then $|\mathcal{F}| \leqslant \binom{n-1}{k-1}$. Moreover, if $n>2k$, then equality holds in the latter if and only if there exists $a\in [n]$ such that $\mathcal{F} = \{ A\subset [n]: |A| = k \mbox{ and } a\in A \}$. These maximum-size families are known as \itbf{stars}, and they are conjugate orbits of the group $\sym{n-1}$ acting on the $k$-subsets of $[n]$. One of the many ways to prove the EKR theorem \cite[Chapter~1]{Godsil2015} is to model intersecting families to correspond to certain structures in a graph. Recall that the Kneser graph $K(n,k)$ is the graph whose vertex set is the collection of all $k$-subsets of $[n]$, and two $k$-subsets $A$ and $B$ are adjacent if and only if $A\cap B = \varnothing$. It is not hard to see that a family $\mathcal{F}$ of $k$-subsets of $[n]$ is intersecting if and only if it is a coclique\footnote{i.e., a subset of vertices in which no two are adjacent.} in $K(n,k)$. Using the fact from the EKR theorem that the maximum cocliques of $K(n,k)$ are stars, it can be shown that 
	\begin{align*}
		\Aut(K(n,k)) = \sym{n}, 
	\end{align*}
	whenever $n\geqslant 2k+1$; see \cite[Corollary~7.8.2]{Godsil2004}.
	
	There are many extensions of the EKR theorem to different combinatorial objects, as described in \cite{Godsil2015} for instance. In particular, the EKR theorem can be extended to permutations of a finite symmetric group. For any $n\geq 3$, we say that two permutations $\sigma$ and $\tau$ of $\sym{n}$ are intersecting if the exists $i\in [n]$ such that $\sigma(i) = \tau (i)$, and a subset $\mathcal{F} \subset \sym{n}$ is called \itbf{intersecting} if any two permutations in $\mathcal{F}$ are intersecting. To find the largest intersecting sets of $\sym{n}$, one can formalize the problem into a graph theoretic one. The \itbf{derangement graph} $\Gamma_n$ of $\sym{n}$ is the graph whose vertex set is $\sym{n}$, and two permutations $\sigma$ and $\tau$ are adjacent if $\sigma^{-1}\tau$ is a derangement, i.e., a fixed-point-free permutation. It is not hard to see that $\Gamma_n$ is in fact the Cayley graph $\operatorname{Cay}(\sym{n}, D_n)$, where $D_n = \left\{ \sigma\in \sym{n}: \sigma(i) \neq i, \mbox{ for all } i\in [n] \right\}$ is the set of all derangements of $\sym{n}$; (see \cite{Godsil2004} for the definition of a Cayley graph). A subset $\mathcal{F} \subset \sym{n}$ is intersecting if and only if $\mathcal{F}$ is a coclique of $\Gamma_n$, i.e., a subset of vertices in which no two are adjacent. It was shown in \cite{Cameron2003,Godsil2009,Larose2004} that the maximum cocliques of $\Gamma_n$ are sets of the form
	\begin{align*}
		S_{i,j} = \left\{ \sigma\in \sym{n}: \sigma(i) = j \right\},
	\end{align*}
	for some $i,j\in [n]$. These sets are exactly the cosets of a point stabilizer of the symmetric group $\sym{n}$. We say that $\sym{n}$ has the \itbf{strict-EKR property} in this case. Using this fact, Deng and Zhang showed in \cite{Deng2011} that 
	$
		\Aut(\Gamma_n) = \left(L_n \rtimes \operatorname{Inn}(\sym{n})\right)\rtimes \langle \phi\rangle,
	$
	where $\phi:\sym{n}\to \sym{n}$ is the automorphism such that $\phi(x) = x^{-1}$ and $L_n = \{ \lambda_\sigma\in \sym{n}^{\sym{n}}: \lambda_\sigma(x) = \sigma x, \mbox{ for }\sigma \in \sym{n} \}$ is the left-regular representation of $\sym{n}$. For $\sigma\in \sym{n}$, let $\rho_{\sigma}:\sym{n} \to \sym{n}$ be such that $\rho_{\sigma}(x) = x\sigma^{-1}$, for $x\in \sym{n}$. It is not hard to show that the automorphism group of $\Gamma_n$ can also be expressed as
	\begin{align*}
		\Aut(\Gamma_n) \cong \left(L_n\times R_n\right)\rtimes \langle \phi \rangle,
	\end{align*}
	where $R_n = \left\{ \rho_{\sigma}:\sigma\in \sym{n} \right\}$ is the right-regular representation of $\sym{n}$.

	The above EKR type problem on permutations can be further generalized to transitive subgroups of $\sym{n}$, i.e., transitive permutation groups. Given a finite transitive permutation group $G\leq \sym{\Omega}$, the \itbf{derangement graph $\Gamma_G$} of $G$ is the Cayley graph $\operatorname{Cay}(G,D_G)$, where $D_G = \{g\in G:g(\omega)\neq \omega, \mbox{ for all }\omega \in \Omega \}$. Then, it is not hard to see that a subset $\mathcal{F}\subset G$ is intersecting if and only if it is a coclique of $\Gamma_G$. We say that $G \leqslant \sym{\Omega}$ has the \itbf{EKR property} if $\alpha(\Gamma_G) = \frac{|G|}{|\Omega|}$. If the only cocliques of maximum size are cosets of a point stabilizer, i.e., subsets of the form $S_{\omega,\omega^\prime} = \{ g\in G: g(\omega) = \omega^\prime \}$ for some $\omega,\omega^\prime \in \Omega$, then we say that $G\leqslant \sym{\Omega}$ has the \itbf{strict-EKR property}. See \cite{Hujdurovic2022,Meagher2021,Pantangi2025} for some recent work on EKR and strict-EKR properties. As in the case of the original EKR theorem and the EKR theorem for the permutations of the symmetric group, it is natural to ask whether the full automorphism group of $\Gamma_G$ can be explicitly computed, provided that $G\leqslant \sym{\Omega}$ has the strict-EKR property. Deducing from computer search using the computer algebra system \verb*|Sagemath| \cite{sagemath}, the answer to this question seems to depend heavily on the structure the underlying permutation group, and not solely on the strict-EKR property. However, for the projective general linear group of degree two over finite fields equipped with its natural geometric action on the projective line, the full automorphism group of the derangement graph can be computed using the strict-EKR property and some geometric properties. We note that Meagher and Spiga showed in \cite{Meagher2011} that this action of $\pgl{2}{q}$ admits the strict-EKR property. We state our main result below.
	\begin{thm}
		For any power $q$ of a prime number, we have
		\begin{align*}
			\Aut(\Gamma_{\pgl{2}{q}}) = \left(L_{\pgl{2}{q}}\times R_{\pgl{2}{q}}\right)\rtimes \left(\langle  \psi \rangle \times \gamma_{\Aut(\mathbb{F}_q)}\right),
		\end{align*}
		where:
		\begin{enumerate}[$\bullet$]
			\item $L_{\pgl{2}{q}}$ is the left-regular representation of $\pgl{2}{q}$,
			\item $R_{\pgl{2}{q}}$ is the right-regular representation of $\pgl{2}{q}$,
			\item $\gamma_{\Aut(\mathbb{F}_q)}$ is the group of conjugation by elements of the field automorphism $\Aut(\mathbb{F}_q)$, and
			\item $\psi:\pgl{2}{q} \to \pgl{2}{q}$ is the anti-automorphism mapping an element to its inverse.
		\end{enumerate}
		\label{thm:main}
	\end{thm}

	A Cayley graph $X$ on a group $G$ is a \itbf{normal Cayley graph} if the group of left-regular representations $L_G$ of $G$ is a normal subgroup of $\Aut(X)$. It is an easy exercise to show that $\psi L_{\pgl{2}{q}} \psi^{-1} = R_{\pgl{2}{q}}$. Therefore, $\psi$ does not normalize $L_{\pgl{2}{q}}$, and we deduce the following. 
	\begin{prop}
		The graph $\Gamma_{\pgl{2}{q}}$ is not a normal Cayley graph.
	\end{prop}
	
	\subsection*{Notations and conventions}
	In the proof of Theorem~\ref{thm:main}, we will adopt the following notations. Let $q = p^k$ be the power of a prime $p$, where $k\geqslant 1$ is an integer. The finite field on $q$ elements is denoted by $\mathbb{F}_q$. We denote the projective line over $\mathbb{F}_q$ by $\pg{1}{q}$, i.e., the set of all $1$-dimensional subspaces of $\mathbb{F}_q^2$. The elements of $\pg{1}{q}$ are the subspaces
	\begin{align*}
		\left\langle 
		\begin{bmatrix}
			1 \\
			x
		\end{bmatrix}
		\right \rangle, \ x\in \mathbb{F}_q
		\mbox{ and }
		\left\langle 
		\begin{bmatrix}
			0\\
			1
		\end{bmatrix}
		\right\rangle .
	\end{align*}
	Thus, we may view $\pg{1}{q}$ as $\mathbb{F}_q$ with an additional point at infinity, i.e., $\pg{1}{q} = \mathbb{F}_q\cup \{\infty\}$.
	
	All groups considered in this paper are finite, and all graphs are simple and undirected. All groups actions are left actions, that its, a group $G$ with identity element $1$ acts on a set $\Omega$ if there exists a map $\alpha:G\times \Omega \to \Omega$ such that for any $g,g^\prime \in G$, we have $\alpha(gg^\prime,\omega) = \alpha(g,\alpha(g^\prime,\omega))$ and $\alpha(1,\omega) =\omega$, for $\omega\in \Omega$. The group $G$ acts on $\Omega$ faithfully if $\{ g\in G: \alpha(g,\omega) = \omega , \forall \omega\in\Omega \} = \{1\}$.
	
	Given a group $G$, we define the maps $\lambda_g: G\to G$ and $\rho_g: G\to G$ such that $\lambda_g(x) = gx$ and $\rho_g(x) = xg^{-1}$, for $x\in G$. The sets $L_G = \{ \lambda_g:g\in G \}$ and $R_G = \{ \rho_g: g\in G \}$ are subgroups of $\sym{G}$, and they are respectively called the \itbf{left-regular representation} and the \itbf{right-regular representation} of the group $G$. If $K$ is an overgroup containing $G$ as a normal subgroup, then for any $k\in K$, let $\gamma_k: G \to G$ such that $\gamma_k(x) = kxk^{-1}$ for $x\in G$. It is not hard to see that $\gamma_k = \lambda_k\rho_k$ for any $k\in K$. The group $L_G$ and $R_G$ commute, and if $Z(G)$ is trivial, then $\langle L_G, R_G \rangle = L_G\times R_G$. See \cite{Herman2025} for details.

	\section{The projective general linear groups of degree two}
		
	The projective general linear group $\pgl{2}{q}$ is the group of all \itbf{M\"obius transformations}, that is, transformations of the form $h: \pg{1}{q} \to \pg{1}{q}$ such that
	\begin{align*}
		h(z) = \frac{a+ bz}{c+dz}
	\end{align*}
	where $ad-cd \neq 0$, and with the rules: $\frac{1}{\infty} = 0$, $\frac{c}{0} = \infty$ for $c\neq 0$ and $h(\infty) = \frac{b}{d}$.

	The group $\pgl{2}{q}$ acts transitively on $\pg{1}{q}$. The stabilizer of $\infty\in \pg{1}{q}$ in this action of $\pgl{2}{q}$ is a Borel subgroup of $\pgl{2}{q}$, which is therefore of order $q(q-1)$, isomorphic to $\agl{1}{q}$. This subgroup acts transitively on the remaining $q$ points of $\pg{1}{q}$, so the action of $\pgl{2}{q}$ on $\pg{1}{q}$ is also $2$-transitive\footnote{that is, it is also transitive on the set of ordered pairs of elements of $\pg{1}{q}$ with distinct entries}. In addition, it is well known that this action of $\pgl{2}{q}$ is sharply $3$-transitive, so for every triples $(x,y,z)$ and $(x^\prime,y^\prime,z^\prime)$ of elements of $\pg{1}{q}$ with distinct entries, there exists a unique $g\in \pgl{2}{q}$ such that $(g(x),g(y),g(z)) = (x^\prime,y^\prime,z^\prime)$.

	Given distinct $x,y,z,t\in \pg{1}{q}$, their \itbf{cross-ratio} is 
	\begin{align*}
		\left(x,y;z,t\right): = \frac{(z-x)(t-y)}{(z-y)(t-x)}.
	\end{align*}
	The cross-ratio is an invariant of $\pgl{2}{q}$ in the sense that if $g\in \pgl{2}{q}$, then $(g(x),g(y);g(z),g(t)) = (x,y;z, t)$ for any distinct $x,y,z,t\in \pg{1}{q}$. In fact, this invariant can be used to determine whether an element of $\sym{\pg{1}{q}}$ is in $\pgl{2}{q}$ or not, as follows.
	
	\begin{thm}
		A permutation of $\pg{1}{q}$ belongs to $\pgl{2}{q} $ if and only if it preserves the cross-ratio of any four distinct elements of $\pg{1}{q}$.\label{thm:cross-ratio}
	\end{thm}
	
		By sharp $3$-transitivity, the group $\pgl{2}{q}$ acts intransitively on the set all quadruples of $\pg{1}{q}$ with distinct entries, whenever $q\geqslant 4$.
	Given two quadruples with distinct entries $(u,v,w,t)$ and $(u^\prime,v^\prime,w^\prime,t^\prime)$ of $\pg{1}{q}$, there exists at most one element of $\pgl{2}{q}$ mapping one to the other. If $g\in \pgl{2}{q}$ is such that $(g(u),g(v),g(w),g(t)) = (u^\prime,v^\prime,w^\prime,t^\prime)$, then by Theorem~\ref{thm:cross-ratio} we have the cross-ratio equality 
	\begin{align*}
		(u,v;w,t) = (u^\prime,v^\prime;w^\prime,t^\prime).
	\end{align*}
	
	Denote the Frobenius automorphism of $\mathbb{F}_q$ by $\theta$, that is, $\theta: \mathbb{F}_q \to \mathbb{F}_q$ such that $\theta(x) = x^p$ for all $x\in \mathbb{F}_q$. It is well known that $\Aut(\mathbb{F}_q) = \Aut(\mathbb{F}_{p^k}/\mathbb{F}_p) = \langle \theta \rangle \cong \mathbb{Z}_k$. The group $\Aut(\mathbb{F}_q)$ acts as automorphism of $\pgl{2}{q}$ via the following action. For any $h: z\to \frac{a+bz}{c+dz}$ in $\pgl{2}{q}$, we have
	\begin{align*}
		\theta\cdot h: z\mapsto \frac{\theta(a) + \theta(b)z}{\theta(c) + \theta(d)z}. 
	\end{align*}
	Then, the projective semilinear group of degree $2$ over $\mathbb{F}_q$ is the group denoted by $\pgammal{2}{q} = \pgl{2}{q} \rtimes \Aut(\mathbb{F}_q) = \pgl{2}{q}\rtimes \langle \theta\rangle$ with multiplication
	\begin{align*}
		(g\theta^i)(h\theta^j) = g (\theta^i\cdot h) \theta^{i+j}
	\end{align*}
	for any $g,h\in \pgl{2}{q}$ and $0\leqslant i,j\leqslant k-1$. We note that by the property of the semi-direct product, we have $\theta \cdot h = \theta h \theta^{-1}$, for any $h\in \pgl{2}{q}$.
	
	Next, we state an important theorem that will be used in the proof of the main result. This result can be derived from the Fundamental Theorem of Projective Geometry \cite[Chapter~2(ii)]{Hirschfeld1998} or from a group theoretic result from \cite[Page~245]{Dixon1996}.
	\begin{thm}
		If $q = p^k$, then 
		$
			\operatorname{N}_{\sym{\pg{1}{q}}}(\pgl{2}{q}) = \pgammal{2}{q}.
		$
		\label{thm:FTPG}
	\end{thm}

	\section{Proof of Theorem~\ref{thm:main}}
	
	In this section, we give a proof to Theorem~\ref{thm:main}. We first show that $\Aut(\Gamma_{\pgl{2}{q}})$ contains a large subgroup isomorphic to $\left(L_{\pgl{2}{q}} \times R_{\pgl{2}{q}}\right) \rtimes \langle \psi, \gamma_\theta\rangle$ in \S~\ref{subsect1}. The rest of the proof is building up to showing that this subgroup is in fact equal to $\Aut(\Gamma_{\pgl{2}{q}})$. In \S~\ref{subsect2}, we show that $\Aut(\Gamma_{\pgl{2}{q}})$ embed into $\sym{\pg{1}{q}\times \pg{1}{q}}$, and use this fact along with the intersection sizes to embed $\Aut(\Gamma_{\pgl{2}{q}})$ into the automorphism group of the $(q+1)\times (q+1)$ grid graph or the Hamming graph $H(2,q+1)$ in \S~\ref{subsect3}. Using all these results, we give a proof to Theorem~\ref{thm:main} in \S~\ref{subsect5}. 
	\subsection{A large subgroup of automorphism}\label{subsect1}
	We will show that $\Aut(\Gamma_{\pgl{2}{q}})$ admits a large subgroup, which turns out to be the full automorphism group.
	
	Since $\Gamma_{\pgl{2}{q}}$ is a (left) Cayley graph, the celebrated result of Sabidussi \cite{Sabidussi1958} asserts that $L_{\pgl{2}{q}}\leqslant \Aut(\Gamma_{\pgl{2}{q}})$. The right-regular representation of need not be contained in the automorphism group of a Cayley graph, in general. However, in the case of $\Gamma_{\pgl{2}{q}}$, the connection set is a union of conjugacy classes, and this implies $R_{\pgl{2}{q}}\leqslant \Aut(\Gamma_{\pgl{2}{q}})$. As $\pgl{2}{q}$ is centreless, we have $\langle L_{\pgl{2}{q}},R_{\pgl{2}{q}} \rangle = L_{\pgl{2}{q}} \times R_{\pgl{2}{q}}$. See \cite[\S~3.1]{Herman2025} for details. The same argument can be used to show that $\langle L_{ \pgammal{2}{q}},R_{\pgammal{2}{q}}\rangle = L_{ \pgammal{2}{q}} \times R_{\pgammal{2}{q}}$, but as we will see later, not every element of this group is an automorphism of $\Gamma_{\pgl{2}{q}}$.

	\begin{lem}
		We have $\Aut (\mathbb{F}_q)\cong \langle \gamma_{\theta}\rangle \leqslant \Aut(\Gamma_{\pgl{2}{q}})$.
	\end{lem}
	\begin{proof}
		It is clear that $\Aut(\mathbb{F}_q) = \langle \theta\rangle$ is isomorphic to $\langle \gamma_{\theta} \rangle$.
		We note that $g\in \pgl{2}{q}$ fixes $u\in \pg{1}{q}$ if and only if $\theta\cdot g$ fixes $\theta(u)$. Hence, if $g\in \pgl{2}{q}$ is a derangement on $\pg{1}{q}$, then so is $\theta\cdot g$.
		For any $x,y\in \pgl{2}{q}$, we have 
		\begin{align*}
			\gamma_\theta(x) \sim \gamma_\theta(y) &\Leftrightarrow \gamma_\theta(x)^{-1} \gamma_\theta(y) \mbox{ has no fixed points in $\pg{1}{q}$}\\
			&\Leftrightarrow \gamma_\theta(x^{-1 }y) = \theta (x^{-1} y) \theta^{-1} = \theta \cdot (x^{-1}y) \mbox{ has no fixed points in $\pg{1}{q}$}\\
			&\Leftrightarrow x^{-1} y \mbox{ has no fixed points in $\pg{1}{q}$}\\
			&\Leftrightarrow x \sim y.
		\end{align*}
		Hence, $\gamma_\theta \in \Aut(\Gamma_{\pgl{2}{q}})$.
	\end{proof}

	We can define an action of $\Aut(\mathbb{F}_q)$ as automorphism on $L_{\pgl{2}{q}}\times R_{\pgl{2}{q}}$ as shown in the next lemma. 
	
	\begin{lem}
		We have 
		\begin{align*}
			\langle L_{\pgl{2}{q}}\times R_{\pgl{2}{q}},\gamma_\theta\rangle = \left(L_{\pgl{2}{q}}\times R_{\pgl{2}{q}}\right)\rtimes \langle \gamma_\theta\rangle.
		\end{align*}
		\label{lem:subgroup}
	\end{lem}
	\begin{proof}
		For any $g,g^\prime \in \pgl{2}{q}$, we have 
		\begin{align*}
			\gamma_\theta (\lambda_g \rho_{g^\prime}) \gamma_\theta^{-1} = (\gamma_\theta \lambda_g \gamma_\theta^{-1}) (\gamma_\theta \rho_{g^\prime} \gamma_\theta^{-1}) = \lambda_{\theta\cdot g} \rho_{\theta \cdot g^\prime}.
		\end{align*}
		Therefore, $L_{\pgl{2}{q}}\times R_{\pgl{2}{q}} \trianglelefteqslant \langle L_{\pgl{2}{q}}\times R_{\pgl{2}{q}},\gamma_\theta \rangle$. If $\phi \in \langle \gamma_\theta\rangle \cap (L_{\pgl{2}{q}}\times R_{\pgl{2}{q}})$ is non-trivial, then there exists $\gamma_\theta^i = \lambda_g \rho_{g^\prime} = \phi$ for some $0\leqslant i \leqslant k-1$ and $g,g^\prime \in \pgl{2}{q}$. Evaluating $\gamma_\theta^i = \lambda_g \rho_{g^\prime}$ on the identity of $\pgl{2}{q}$, we have $g^\prime = g$, and consequently, $\lambda_g\rho_{g^\prime} = \lambda_g\rho_g = \gamma_g$. We deduce that $\gamma_\theta^i = \gamma_{\theta^i} = \gamma_g$. As $\pgl{2}{q} \cap \Aut(\mathbb{F}_q)$ is trivial, we deduce that $i = 0$ and $g = 1$, and thus $\phi$ is the identity element.
		Consequently, $\langle L_{\pgl{2}{q}}\times R_{\pgl{2}{q}},\gamma_\theta \rangle = (L_{\pgl{2}{q}}\times R_{\pgl{2}{q}}) \rtimes \langle \gamma_\theta \rangle$.
	\end{proof}

	Recall that $\psi: \pgl{2}{q} \to \pgl{2}{q}$ is the map such that $\psi(x) = x^{-1}$ for any $x\in \pgl{2}{q}$. Since the centre of $\pgl{2}{q}$ is trivial, it follows from the result in \cite[\S~3.1]{Herman2025} that $\langle L_{\pgl{2}{q}}\times R_{\pgl{2}{q}},\psi \rangle = \left(L_{\pgl{2}{q}}\times R_{\pgl{2}{q}}\right) \rtimes \langle \psi\rangle$. Moreover, for $x\in \pgl{2}{q}$, we have
	\begin{align*}
		\theta \psi(x) = \theta(x^{-1}) = \left(x^{-1}\right)^p = \left(x^p\right)^{-1} = \psi (x^p) = \psi \theta(x).
	\end{align*} 
	In other words, $\psi$ and $\theta$ commute. We deduce the following lemma whose proof is omitted.
	\begin{lem}
		For any $q$, we have 
		\begin{align*}
			\left(L_{\pgl{2}{q}}\times R_{\pgl{2}{q}}\right) \rtimes  \left(\langle \psi \rangle \times \langle \gamma_\theta \rangle\right) \leqslant \Aut(\Gamma_{\pgl{2}{q}}).
		\end{align*}
		\label{lem:part1}
	\end{lem}
	
	For the remainder of the section, we show that equality in fact holds in Lemma~\ref{lem:part1}.

	\subsection{An induced action on $\pg{1}{q}\times \pg{1}{q}$}\label{subsect2} In \cite{Meagher2011}, it was shown that the permutation group $\pgl{2}{q}$ with its action on $\pg{1}{q}$ admits the strict-EKR property. For any $u,v\in \pg{1}{q}$, we define 
	\begin{align}
		S_{u,v}:= \left\{ g\in \pgl{2}{q}:g(u) = v \right\}.
	\end{align}
	We note that $S_{u,u}$ is the stabilizer of $u\in \pg{1}{q}$ in $\pgl{2}{q}$, and for any $v\in \pg{1}{q}$, the set $S_{u,v}$ is a coset of $S_{u,u}$.

	For any triples $(u_1,u_2,u_3)$ and $(v_1,v_2,v_3)$ of $\pg{1}{q}$ with distinct entries, we have
	\begin{align*}
		|S_{u_1,v_1}\cap S_{u_2,v_2}\cap S_{u_3,v_3}| = 1
	\end{align*}
	due to the fact that $\pgl{2}{q}$ is sharply $3$-transitive. As discussed before, $\pgl{2}{q}$ acts intransitively on the set of quadruple of elements of $\pg{1}{q}$ with distinct entries, for $q\geqslant 4$, and two quadruples $(u_1,u_2,u_3,u_4)$ and $(v_1,v_2,v_3,v_4)$ of elements of $\pg{1}{q}$ with distinct entries belong to the same orbit if and only if they have the same cross-ratio. Therefore, we deduce that 
	\begin{align*}
		|S_{u_1,v_1}\cap S_{u_2,v_2}\cap S_{u_3,v_3}\cap S_{u_4,v_4}| = 1 \mbox{ if and only if }(u_1,u_2;u_3,u_4) = (v_1,v_2;v_3,v_4).
	\end{align*}

	Let $\Omega = \left\{ S_{u,v}: u,v\in \pg{1}{q} \right\}$. As $\Omega$ is the set of all cocliques of maximum size in $\Gamma_{\pgl{2}{q}}$, the automorphism group $\Aut(\Gamma_{\pgl{2}{q}})$ acts on $\Omega$. 
	\begin{lem}
		The action of $\Aut(\Gamma_{\pgl{2}{q}})$ on $\Omega$ is faithful.\label{lem:faithful1}
	\end{lem}
	\begin{proof}
		Assume that $\phi \in \Aut(\Gamma_{\pgl{2}{q}})$ such that $\phi(S_{u,v}) = S_{u,v}$ for all $u,v\in \pg{1}{q}$. Then, given $g\in \pgl{2}{q}$ we have 
		\begin{align*}
			\phi(\{g \}) = \phi\left(\bigcap_{u\in \pg{1}{q}} S_{u,g(u)}\right) \subseteq \bigcap_{u\in \pg{1}{q}} \phi(S_{u,g(u)}) = \bigcap_{u\in \pg{1}{q}} S_{u,g(u)} = \{g\}.
		\end{align*}
		Hence, $\phi(g) = g$ for all $g\in \pgl{2}{q}$. Since $\Aut(\Gamma_{\pgl{2}{q}})$ acts faithfully on the vertices of $\Gamma_{\pgl{2}{q}}$, we must have that $\phi$ is the identity map.
	\end{proof}
	Thus, we have the embedding $\Aut(\Gamma_{\pgl{2}{q}}) \hookrightarrow \sym{\Omega} \cong \sym{\pg{1}{q} \times \pg{1}{q}}$.
	Next, we examine the intersection of elements of $\Omega$. If $\phi \in \Aut(\Gamma_{\pgl{2}{q}})$, then 
	\begin{align}
		|\phi(S_{u,v})\cap \phi (S_{u^\prime,v^\prime})| = |S_{u,v}\cap S_{u^\prime,v^\prime}| \label{eq:constant-intersection-size}
	\end{align}
	for any $u,v,u^\prime,v^\prime \in \pg{1}{q}$. Note that the latter holds more generally for permutations of $\pgl{2}{q}$ that also permute $\Omega$. 
	
	As previously discussed, the left-regular representation $L_{\pgl{2}{q}} = \left\{ \lambda_g: g\in \pgl{2}{q} \right\}$ of $\pgl{2}{q}$ is a subgroup of $\Aut(\Gamma_{\pgl{2}{q}}).$ Its action on $\Omega$ is as follows. For $g\in \pgl{2}{q}$ and $u,v\in \pg{1}{q}$, we have 
	$
		\lambda_g(S_{u,v}) =  gS_{u,v} = S_{u,g(v)}.
	$
	Similarly, for $g\in \pgl{2}{q}$, we have $\rho_g(S_{u,v}) = S_{u,v}g^{-1} = S_{g(u),v}$.
	
	We determine the intersection sizes of elements of $\Omega$ in the next lemma.
	\begin{lem}
		If $u,u^\prime,v,v^\prime \in \pg{1}{q}$, then
		\begin{align*}
			|S_{u,v}\cap S_{u^\prime,v^\prime}|
			=
			\begin{cases}
				q(q-1)& \mbox{ if $u = u^\prime$ and $v = v^\prime$}\\
				0& \mbox{ if $u = u^\prime$ and $v\neq v^\prime$, or $u\neq u^\prime$ and $v = v^\prime$} \\ 
				q-1 & \mbox{ if $(u,v)\neq (u^\prime,v^\prime)$}.
			\end{cases}
		\end{align*}
		\label{lem:intersections}
	\end{lem}
	\begin{proof}
		Clearly, if $u = u^\prime$ and $v = v^\prime$, then $S_{u,v} = S_{u^\prime,v^\prime}$, so the size of their intersection is $q(q-1)$. The intersections $ S_{u,v}\cap S_{u,v^\prime}$ where $v\neq v^\prime$ and $S_{u,v}\cap S_{u^\prime,v}$ where $u\neq u^\prime$ are both empty, otherwise the elements would not be a map in the former, and would not be injective in the latter. Assume next that $(u,v)\neq (u^\prime,v^\prime)$. Let $g\in \pgl{2}{q}$ such that $g(v) = u$. Then, by \eqref{eq:constant-intersection-size} we have
		\begin{align}
			\left|S_{u,v}\cap S_{u^\prime,v^\prime} \right| &= \left|\lambda_g(S_{u,v}) \cap \lambda_g(S_{u^\prime,v^\prime})\right|
			= \left|S_{u,g(v)}\cap S_{u^\prime,g(v^\prime)}\right|
			= \left|S_{u,u} \cap  S_{u^\prime,g(v^\prime)}\right|.\label{eq:coset-intersections}
		\end{align}
		Since $\pgl{2}{q}$ is $2$-transitive on $\pg{1}{q}$, the stabilizer $S_{u,u}$ of $u\in \pg{1}{q}$ has two orbits, namely $\{u\}$ and $\pg{1}{q} \setminus \{u\}$. Note that since $(u,v)\neq (u^\prime,v^\prime)$, we know that $u^\prime, g(v^\prime) \in \pg{1}{q}\setminus\{u\}$. Hence, there exists $h\in S_{u,u}$ such that $h(g(v^\prime)) = u^\prime$. Using this in \eqref{eq:coset-intersections}, we have 
		\begin{align*}
			\left|S_{u,v}\cap S_{u^\prime,v^\prime}\right| = |S_{u,u}\cap S_{u^\prime,g(v^\prime)}| = |hS_{u,u}\cap hS_{u^\prime,g(v^\prime)}| = |S_{u,u} \cap S_{u^\prime,u^\prime}|.
		\end{align*}
		Note that $S_{u,u}\cap S_{u^\prime,u^\prime}$ is the pointwise stabilizer of $u$ and $u^\prime$. As $S_{u,u}$ acts transitively on $\pg{1}{q}\setminus \{u\}$, by the orbit stabilizer lemma we have 
		\begin{align*}
			|S_{u,u} \cap S_{u^\prime,u^\prime}| = \frac{|S_{u,u}|}{|\pg{1}{q}\setminus \{u\}|} = \frac{q(q-1)}{q} = q-1.
		\end{align*}
		Hence, we deduce that $|S_{u,v}\cap S_{u^\prime,v^\prime}| = q-1$. 
	\end{proof}
	
	\subsection{Reduction to the $(q+1)\times (q+1)$--grid}\label{subsect3}
	For any integers $d,q\geq 1$, the Hamming graph $H(d,n)$ is the graph whose vertex set consists of the $d$-tuples with entries from $\{0,1,\ldots,q-1\}$, and where two $d$-tuples are adjacent if they differ in exactly one entry or position. The graph $H(2,n)$ is also known as the $n\times n$ grid graph.
	
	From Lemma~\ref{lem:intersections}, we know that the intersection sizes take exactly three values. From this, we may define a graph that represents the structure of the intersection of elements of $\Omega$. Let $\mathcal{G}_q$ be the graph whose vertex set is $\pg{1}{q} \times \pg{1}{q}$, where two vertices $(u,v)$ and $(u^\prime,v^\prime)$ are adjacent if $|S_{u,v}\cap S_{u^\prime,v^\prime}| = 0$. We illustrate $\mathcal{G}_3$ below. 
	\begin{figure}[H]
		\centering
		\includegraphics[width=7cm]{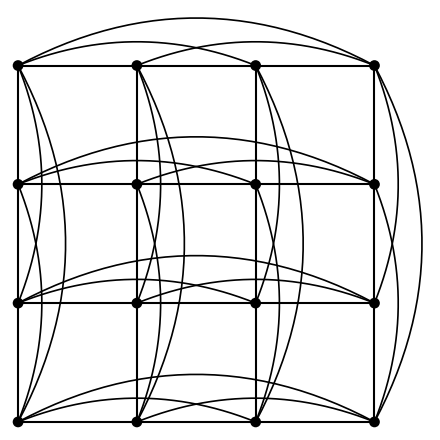}
		\caption{The graph $\mathcal{G}_{3}$.}
	\end{figure}
	
	As the group $\Aut(\Gamma_{\pgl{2}{q}})$ acts (faithfully) on $\Omega$, it also acts (faithfully) on the vertices $\mathcal{G}_q$. To obtain this action, we view any $\phi \in \Aut(\Gamma_{\pgl{2}{q}}) \hookrightarrow \sym{\Omega}$ as a mapping on $\pg{1}{q}\times \pg{1}{q}$ to itself. We write
	\begin{align}
		\phi(u,v) = (u^\prime,v^\prime) \Leftrightarrow \phi(S_{u,v}) = S_{u^\prime,v^\prime}\label{eq:equivalence}
	\end{align}
	for any $u,v,v^\prime,v^\prime \in \pg{1}{q}$ and $\phi \in \Aut(\Gamma_{\pgl{2}{q}})$. In fact, $\Aut(\Gamma_{\pgl{2}{q}})$ does not only act faithfully on $\pg{1}{q}\times \pg{1}{q}$, it also acts as automorphism of $\mathcal{G}_q$.
	\begin{lem}
		Any $\phi \in \Aut(\Gamma_{\pgl{2}{q}})$ induces an automorphism of $\mathcal{G}_q$. In particular, $\Aut(\Gamma_{\pgl{2}{q}})$ embeds into $\Aut(\mathcal{G}_q)$.\label{lem:induced-automorphism}
	\end{lem}
	\begin{proof}
		By Lemma~\ref{lem:faithful1}, $\phi$ induces a permutation on $\pg{1}{q} \times \pg{1}{q}$. Given two vertices $(u,v)$ and $(u^\prime,v^\prime)$ of $\mathcal{G}_q$, we have 
		\begin{align*}
			(u,v) \sim_{\mathcal{G}_q} (u^\prime,v^\prime) &\Leftrightarrow S_{u,v} \cap S_{u^\prime,v^\prime} = \varnothing\\
			&\Leftrightarrow \phi (S_{u,v}) \cap \phi(S_{u^\prime,v^\prime}) = \varnothing   \qquad (\mbox{ by }\eqref{eq:constant-intersection-size})\\
			&\Leftrightarrow \phi(u,v) \sim_{\mathcal{G}_q} \phi(u^\prime,v^\prime).
		\end{align*}
		In other words, $\phi$ induces an automorphism of $\mathcal{G}_q$. The rest of the proof is immediate.
	\end{proof}
	
	\begin{table}[hb]
		\begin{tabular}{c|c|c}
			$\phi\in \Aut(\Gamma_{\pgl{2}{q}})$& $\overline{\phi}\in \Aut(\mathcal{G}_q)$&Comments\\
			\hline\hline
			$\lambda_g$& $(1,g;1)$&$g\in \pgl{2}{q}$\\
			$\rho_g$& $(g,1;1)$&$g\in \pgl{2}{q}$\\
			$\psi$ & $(1,1;-1)$ & swaps coordinates\\
			$\gamma_\theta$ & $(\theta,\theta;1)$&$\theta \in \Aut(\mathbb{F}_q)$\\
			\hline
		\end{tabular}
		\caption{Images of certain automorphisms of $\Aut(\Gamma_{\pgl{2}{q}}) \hookrightarrow \Aut(\mathcal{G}_q)$.}\label{tab}
	\end{table}
	
	Using Lemma~\ref{lem:intersections}, the graph $\mathcal{G}_q$ is clearly isomorphic the Hamming graph $H(2,q+1)  = K_{q+1} \square K_{q+1} $.
	Consequently, we have $\Aut(\mathcal{G}_q) = \sym{\pg{1}{q}}\wr C_{2} = (\sym{\pg{1}{q}} \times \sym{\pg{1}{q}})\rtimes C_2$, which consists of all triples
			\begin{align*}
				(\sigma,\tau;r)\in \sym{\pg{1}{q}}\times \sym{\pg{1}{q}}\times C_2
			\end{align*}
			such that 
			\begin{align*}
				(\sigma,\tau;r)(i,j) =
				\begin{cases}
					(\sigma(i),\tau(j)) & \mbox{ if }r = 1\\
					(\tau(j),\sigma(i)) & \mbox{ if } r= -1.
				\end{cases}
			\end{align*}

		As $\sym{\pg{1}{q}}$ is centreless, we have
		\begin{align*}
			\langle L_{\sym{\pg{1}{q}}},R_{\sym{\pg{1}{q}}},\psi\rangle = \left(L_{\sym{\pg{1}{q}}} \times R_{\sym{\pg{1}{q}}}\right) \rtimes \langle \psi \rangle.
		\end{align*}
		The group $\left(L_{\sym{\pg{1}{q}}} \times R_{\sym{\pg{1}{q}}}\right) \rtimes \langle \psi \rangle$ acts on $\Omega$ by permuting the indices in $\Omega$.
		Define the map 
		\begin{align*}
			\Phi:\left(L_{\sym{\pg{1}{q}}} \times R_{\sym{\pg{1}{q}}}\right) \rtimes \langle \psi \rangle &\longrightarrow (\sym{\pg{1}{q}} \times \sym{\pg{1}{q}})\rtimes C_2\\
			\lambda_\sigma\rho_{\tau}\psi^i &\longmapsto(\tau,\sigma;(-1)^i).
		\end{align*}
		The map $\Phi$ is clearly a group isomorphism.
		Let $\Phi(\phi) = \overline{\phi}$ is the permutation induced by $\phi\in (L_{\sym{\pg{1}{q}}} \times R_{\sym{\pg{1}{q}}})\rtimes \langle \psi \rangle$ in $\pg{1}{q}\times \pg{1}{q}$. For any subgroup $M\leqslant (L_{\sym{\pg{1}{q}}} \times R_{\sym{\pg{1}{q}}}) \rtimes \langle \psi \rangle$, we define $\overline{M} = \Phi(M)$. We give the image under $\Phi$ of certain elements of $\Aut(\Gamma_{\pgl{2}{q}})$ in its embedding into $\Aut(\mathcal{G}_q)$ in Table~\ref{tab}; the proof is straightforward, so we omit it.

	We note that any $(\sigma,\tau;(-1)^i)\in \sym{\pg{1}{q}}\wr C_2$ can be expressed as the product
	\begin{align*}
		(\sigma,\tau;(-1)^i) = (\sigma,\tau;1)(1,1;(-1)^i) = (\sigma,\tau;1)\overline{\psi}^i.
	\end{align*}
	Since $\psi \in \Aut(\Gamma_{\pgl{2}{q}})$, it is an easy exercise to show that there exists a subgroup $K\leqslant \sym{\pg{1}{q}} \times \sym{\pg{1}{q}} \leq \sym{\pg{1}{q}}\wr C_2$ such that 
	\begin{align}
		\Aut(\Gamma_{\pgl{2}{q}})\cong \overline{\Aut(\Gamma_{\pgl{2}{q}})} = K\langle \overline{\psi}\rangle =  K\rtimes \langle \overline{\psi}\rangle.\label{eq:K}
	\end{align}
	In particular, 
	\begin{align}
		\Aut(\Gamma_{\pgl{2}{q}}) = \Phi^{-1}(K)\rtimes \langle \psi\rangle.\label{eq:Aut-form}
	\end{align}
		
	\subsection{Proof of Theorem~\ref{thm:main}}\label{subsect5}
	In this section, we combine the results from the previous sections to determine the subgroup $\Phi^{-1}(K)$, or equivalently, $K\leqslant \sym{\pg{1}{q}} \times \sym{\pg{1}{q}}.$ 	
	
	For any $g\in \pgl{2}{q}$, recall from Table~\ref{tab} that $\overline{\lambda_g} = (1,g;1)$ and $\overline{\rho_g} = (g,1;1)$. Moreover, we have $\overline{\gamma_\theta} = (\theta,\theta;1)$.
	It is therefore clear from Lemma~\ref{lem:subgroup} that
	\begin{align}
		(\overline{L}_{\pgl{2}{q}}\times \overline{R}_{\pgl{2}{q}})\rtimes \langle \overline{\gamma_\theta}\rangle\leqslant K,\label{eq:subgroup}
	\end{align} 
	where $K$ is the subgroup defined in \eqref{eq:K}.
	It is worth noting that a typical element of $(\overline{L}_{\pgl{2}{q}}\times \overline{R}_{\pgl{2}{q}}) \rtimes \langle \overline{\gamma_\theta} \rangle$ is of the form $(g\theta^i,g^\prime \theta^i;0)$, where $g,g^\prime \in \pgl{2}{q}$ and $0\leqslant i,i^\prime \leqslant k-1$.

	Now, let $\phi \in \Aut(\Gamma_{\pgl{2}{q}})$ such that $\overline{\phi} \in K$. As $K\leq \sym{\pg{1}{q}}\times \sym{\pg{1}{q}}\leqslant \sym{\pg{1}{q}}\wr C_2$, there exist two permutations $\sigma,\tau \in \sym{\pg{1}{q}}$ such that 
	\begin{align*}
		\overline{\phi} = \left(\sigma,\tau;1\right).
	\end{align*}
	\begin{claim}
		The permutations $\sigma $ and $\tau$ lie in $\pgammal{2}{q}$.\label{claim1}
	\end{claim}
	\begin{proof}[Proof Claim~\ref{claim1}]
		For any $\varphi \in \Aut(\Gamma_{\pgl{2}{q}})$ and for any pair of quadruples $(u_1,u_2,u_3,u_4)$ and $(v_1,v_2,v_3,v_4)$ with distinct entries from $\pg{1}{q}$, we have 
		\begin{align*}
			|S_{u_1,v_1} \cap S_{u_2,v_2} \cap S_{u_3,v_3}\cap S_{u_4,v_4}| &= |\varphi(S_{u_1,v_1}) \cap \varphi(S_{u_2,v_2}) \cap \varphi(S_{u_3,v_3}) \cap \varphi(S_{u_4,v_4})|.
		\end{align*}
		Let us show that $\sigma g\sigma ^{-1},\tau g\tau^{-1} \in \pgl{2}{q}$, for any $g\in \pgl{2}{q}$. For any four distinct elements $u_1,u_2,u_3$ and $u_4$ of $\pg{1}{q}$, using the fact that ${\phi}$ and ${\rho_g}$ are automorphisms of $\Gamma_{\pgl{2}{q}}$, we have 
		\begin{align*}
			1 &= \left|\bigcap_{i = 1}^4 S_{u_i,u_i}\right| =\left|\bigcap_{i = 1}^4 {\phi^{-1}} (S_{u_i,u_i})\right| =\left|\bigcap_{i = 1}^4  (S_{\overline{\phi^{-1}}(u_i,u_i)})\right|= \left|\bigcap_{i = 1}^4 S_{\sigma^{-1}(u_i),\tau^{-1}(u_i)}\right|\\
			&\hspace*{1cm}= \left|\bigcap_{i = 1}^4 {\rho_g}(S_{\sigma^{-1}(u_i),\tau^{-1}(u_i)})  \right|= \left|\bigcap_{i = 1}^4 S_{g\sigma^{-1}(u_i),\tau^{-1}(u_i)} \right|= \left|\bigcap_{i = 1}^4 {\phi}(S_{g\sigma^{-1}(u_i),\tau^{-1}(u_i)})\right|\\
			&\hspace*{1.5cm}= \left|\bigcap_{i = 1}^4 (S_{\overline{\phi}(g\sigma^{-1}(u_i),\tau^{-1}(u_i))})\right|= \left|\bigcap_{i = 1}^4 S_{\sigma g\sigma^{-1}(u_i),\tau \tau^{-1}(u_i)}\right|= \left|\bigcap_{i = 1}^4 S_{\sigma g\sigma^{-1}(u_i),u_i}\right|.
		\end{align*}
		Hence, there is a unique element of $\pgl{2}{q}$ that maps the quadruples $(u_1,u_2,u_3,u_4)$ to $\left(\sigma g\sigma^{-1}(u_1),\sigma g\sigma^{-1}(u_2),\sigma g\sigma^{-1}(u_3),\sigma g \sigma^{-1} (u_4)\right)$. By Theorem~\ref{thm:cross-ratio}, we must have 
		\begin{align*}
			\left(u_1,u_2;u_3,u_4\right) = \left(\sigma g\sigma^{-1}(u_1), \sigma g \sigma^{-1}(u_2); \sigma g\sigma^{-1}(u_3),\sigma g\sigma^{-1}(u_4)\right),
		\end{align*}
		for any quadruples $(u_1,u_2,u_3,u_4)$ of $\pg{1}{q}$ with distinct entries. In other words, $\sigma g\sigma^{-1} \in \pgl{2}{q}$ for any $g\in \pgl{2}{q}$. Thus, $\sigma \in \operatorname{N}_{\sym{\pg{1}{q}}}(\pgl{2}{q})$. Similarly, we can show that $\tau \in \operatorname{N}_{\sym{\pg{1}{q}}}(\pgl{2}{q})$ by using the automorphism $\lambda_g\in L_{\pgl{2}{q}}$. The result now follows from Theorem~\ref{thm:FTPG}.
	\end{proof}
	
	From Claim~\ref{claim1}, we deduce that $K \leqslant \overline{L}_{\pgammal{2}{q}} \times \overline{R}_{\pgammal{2}{q}}$. Therefore, for $\varphi \in \Aut(\Gamma_{\pgl{2}{q}})$, we have 
	\begin{align*}
		\overline{\varphi} = (g,g^\prime;(-1)^i) = \overline{\lambda_{g^\prime}}\overline{\rho_g}\overline{\psi^i}
	\end{align*} 
	for some $g,g^\prime \in \pgammal{2}{q}$ and $i\in \{0,1\}$. Since $\Phi$ is an isomorphism, we note that $\Phi^{-1}(K) \leq L_{\pgammal{2}{q}}\times R_{\pgammal{2}{q}}$, and therefore we have
	\begin{align}
		\left(L_{\pgl{2}{q}} \times R_{\pgl{2}{q}}\right) \rtimes \langle \gamma_\theta\rangle \leqslant \Phi^{-1}(K) \leqslant \left(L_{\pgammal{2}{q}} \times R_{\pgammal{2}{q}}\right).\label{eq:inclusions}
	\end{align}
	
	Next, we determine the elements of $L_{\pgammal{2}{q}} \times R_{\pgammal{2}{q}}$  that lie in $\Aut(\Gamma_{\pgl{2}{q}})$.

	\begin{claim}
		Let $g,g^\prime \in \pgl{2}{q}$ and $0\leqslant i,j\leqslant k-1$. Then, $\rho_{g^\prime \theta^j}\lambda_{g\theta^i} \in \Aut(\Gamma_{\pgl{2}{q}})$ if and only if $i = j$. In this case, $\rho_{g^\prime \theta^j}\lambda_{g\theta^i} = \rho_g\lambda_{g^\prime}\gamma_{\theta^i} \in (L_{\pgl{2}{q}}\times R_{\pgl{2}{q}}) \rtimes \langle \gamma_\theta\rangle$.\label{claim2}
	\end{claim}
	\begin{proof}
		The converse of the statement follows from Lemma~\ref{lem:part1}. Now, assume that $\rho_{g^\prime \theta^j}\lambda_{g\theta^i} \in \Aut(\Gamma_{\pgl{2}{q}})$. It is immediate that
		\begin{align*}
			\rho_{g^\prime \theta^j}\lambda_{g\theta^i} \in \Aut(\Gamma_{\pgl{2}{q}}) \mbox{ if and only if } \rho_{\theta^j}\lambda_{\theta^i} \in \Aut(\Gamma_{\pgl{2}{q}}).
		\end{align*}
		As $\rho_{\theta^j}\lambda_{\theta^i} \in \Aut(\Gamma_{\pgl{2}{q}})$, it is a permutation of $\pgl{2}{q}$. In particular, $\rho_{\theta^j}\lambda_{\theta^i} (x) \in \pgl{2}{q}$ for any $x \in \pgl{2}{q}$. As $\rho_{\theta^j}\lambda_{\theta^i}(x) =  \theta^ix \theta^{-j}$, for $x\in \pgl{2}{q}$, we have $\theta^i \pgl{2}{q} \theta^{-j} \subseteq \pgl{2}{q}$. The latter can only happen if $i=j$. \qedhere
	\end{proof}
	
	We deduce from Claim~\ref{claim2} that the elements of  $\Phi^{-1}(K)$ must be of the form 
	\begin{align*}
		\lambda_{g\theta^i} \rho_{g^\prime \theta^i} = \lambda_g \rho_{g^\prime} \lambda_{\theta^i}\rho_{\theta^i} = \lambda_g \rho_{g^\prime} \gamma_{\theta^i} \in (L_{\pgl{2}{q}}\times R_{\pgl{2}{q}}) \rtimes \langle \gamma_{\theta} \rangle
	\end{align*}
	for some $g,g^\prime \in \pgl{2}{q}$ and $0\leqslant i\leqslant k-1$. By \eqref{eq:inclusions}, we have  
	\begin{align*}
		\Phi^{-1}(K) = (L_{\pgl{2}{q} \times R_{\pgl{2}{q}}}) \rtimes \langle \gamma_{\theta}\rangle.
	\end{align*}
	Therefore, we conclude that
	\begin{align*}
		\Aut(\Gamma_{\pgl{2}{q}}) = \Phi^{-1}(K)\rtimes \langle \psi\rangle = (L_{\pgl{2}{q}} \times R_{\pgl{2}{q}}) \rtimes \langle \gamma_\theta,\psi\rangle,
	\end{align*}
	thus proving Theorem~\ref{thm:main}.
	
	\bibliography{ref}

@Article{Godsil1983,
  author    = {Godsil, C.D.},
  journal   = {European J.Combin.},
  title     = {The Automorphism Groups of Some Cubic {C}ayley Graphs},
  year      = {1983},
  issn      = {0195-6698},
  month     = Mar,
  number    = {1},
  pages     = {25--32},
  volume    = {{\bf 4}},
  doi       = {10.1016/s0195-6698(83)80005-5},
  publisher = {Elsevier BV},
}

@Article{Imrich1976,
  author    = {Imrich, W. and Watkins, M. E.},
  journal   = {Period. Math. Hungar.},
  title     = {On automorphism groups of {C}ayley graphs},
  year      = {1976},
  issn      = {1588-2829},
  month     = Sept,
  number    = {3-4},
  pages     = {243--258},
  volume    = {{\bf 7}},
  doi       = {10.1007/bf02017943},
  publisher = {Springer Science and Business Media LLC},
}

@Article{Fang2002,
  author    = {Fang, Xin Gui and Praeger, Cheryl E. and Wang, Jie},
  journal   = {J. Lond. Math. Soc.},
  title     = {On the Automorphism Groups of {C}ayley Graphs of Finite Simple Groups},
  year      = {2002},
  issn      = {0024-6107},
  month     = Dec,
  number    = {3},
  pages     = {563--578},
  volume    = {{\bf 66}},
  doi       = {10.1112/s0024610702003666},
  publisher = {Wiley},
}

@Article{Ganesan2013,
  author    = {Ganesan, Ashwin},
  journal   = {Discrete Math.},
  title     = {Automorphism groups of {C}ayley graphs generated by connected transposition sets},
  year      = {2013},
  issn      = {0012-365X},
  month     = Nov,
  number    = {21},
  pages     = {2482--2485},
  volume    = {{\bf 313}},
  doi       = {10.1016/j.disc.2013.07.013},
  publisher = {Elsevier BV},
}

@Article{Feng2006,
  author    = {Feng, Yan-Quan},
  journal   = {J. Combin. Theory Ser. B},
  title     = {Automorphism groups of {C}ayley graphs on symmetric groups with generating transposition sets},
  year      = {2006},
  issn      = {0095-8956},
  month     = Jan,
  number    = {1},
  pages     = {67--72},
  volume    = {{\bf 96}},
  doi       = {10.1016/j.jctb.2005.06.010},
  publisher = {Elsevier BV},
}

@Book{Godsil2004,
  author    = {Godsil, Chris and Royle, Gordon},
  publisher = {Springer},
  title     = {Algebraic graph theory},
  year      = {2004},
  address   = {New York},
  isbn      = {9780387952208},
  number    = {207},
  series    = {Graduate texts in mathematics},
  pagetotal = {439},
  ppn_gvk   = {548571503},
}

@Book{Godsil2015,
  author    = {Godsil, Christopher and Meagher, Karen},
  publisher = {Cambridge University Press},
  title     = {Erdős–Ko–Rado Theorems: Algebraic Approaches},
  year      = {2015},
  isbn      = {9781316414958},
  month     = Nov,
  doi       = {10.1017/cbo9781316414958},
}

@Article{Larose2004,
  author    = {Larose, Benoit and Malvenuto, Claudia},
  journal   = {European J. Combin.},
  title     = {Stable sets of maximal size in {K}neser-type graphs},
  year      = {2004},
  issn      = {0195-6698},
  month     = July,
  number    = {5},
  pages     = {657--673},
  volume    = {{\bf 25}},
  doi       = {10.1016/j.ejc.2003.10.006},
  publisher = {Elsevier BV},
}

@Article{Godsil2009,
  author    = {Godsil, Chris and Meagher, Karen},
  journal   = {European J. Combin.},
  title     = {A new proof of the {E}rdős–{K}o–{R}ado theorem for intersecting families of permutations},
  year      = {2009},
  issn      = {0195-6698},
  month     = Feb,
  number    = {2},
  pages     = {404--414},
  volume    = {{\bf 30}},
  doi       = {10.1016/j.ejc.2008.05.006},
  publisher = {Elsevier BV},
}

@Article{Cameron2003,
  author    = {Cameron, Peter J. and Ku, C.Y.},
  journal   = {European J. Combin.},
  title     = {Intersecting families of permutations},
  year      = {2003},
  issn      = {0195-6698},
  month     = Oct,
  number    = {7},
  pages     = {881--890},
  volume    = {{\bf 24}},
  doi       = {10.1016/s0195-6698(03)00078-7},
  publisher = {Elsevier BV},
}

@Article{Deng2011,
  author    = {Deng, Yun-Ping and Zhang, Xiao-Dong},
  journal   = {Electron. J. Combin.},
  title     = {Automorphism {G}roup of the {D}erangement {G}raph},
  year      = {2011},
  issn      = {1077-8926},
  month     = Oct,
  number    = {1},
  volume    = {{\bf 18}},
  doi       = {10.37236/685},
  publisher = {The Electronic Journal of Combinatorics},
}

@Article{Meagher2011,
  author    = {Meagher, Karen and Spiga, Pablo},
  journal   = {J. Combin. Theory Ser. A},
  title     = {An {E}rdős–{K}o–{R}ado theorem for the derangement graph of ${PGL}(2,q)$ acting on the projective line},
  year      = {2011},
  issn      = {0097-3165},
  month     = Feb,
  number    = {2},
  pages     = {532--544},
  volume    = {{\bf 118}},
  doi       = {10.1016/j.jcta.2010.11.003},
  publisher = {Elsevier BV},
}

@Article{Herman2025,
  author    = {Herman, Allen and Maleki, Roghayeh and Razafimahatratra, Andriaherimanana Sarobidy},
  journal   = {J. Combin. Des.},
  title     = {On the {T}erwilliger Algebra of the Group Association Scheme of the Symmetric Group ${S}ym(7)$},
  year      = {2025},
  issn      = {1520-6610},
  month     = Apr,
  number    = {7},
  pages     = {261--274},
  volume    = {{\bf 33}},
  doi       = {10.1002/jcd.21981},
  publisher = {Wiley},
}

@Article{Sabidussi1958,
  author    = {Sabidussi, Gert},
  journal   = {Proc. Amer. Math. Soc.},
  title     = {On a class of fixed-point-free graphs},
  year      = {1958},
  issn      = {0002-9939},
  month     = Oct,
  number    = {5},
  pages     = {800--804},
  volume    = {{\bf 9}},
  doi       = {10.1090/s0002-9939-1958-0097068-7},
  publisher = {American Mathematical Society (AMS)},
}

@Book{Dixon1996,
  author    = {John D. Dixon and Brian Mortimer},
  publisher = {Springer New York},
  title     = {Permutation Groups},
  year      = {1996},
  doi       = {10.1007/978-1-4612-0731-3},
  issn      = {0072-5285},
  journal   = {Graduate Texts in Mathematics},
}

@Article{ERDOS1961,
  author    = {Erd\H{o}s, P. and Ko, Chao and Rado, R.},
  journal   = {Q. J. Math.},
  title     = {Intersection theorems for systems of finite sets},
  year      = {1961},
  issn      = {1464-3847},
  number    = {1},
  pages     = {313--320},
  volume    = {{\bf 12}},
  doi       = {10.1093/qmath/12.1.313},
  publisher = {Oxford University Press (OUP)},
}

@Book{Hirschfeld1998,
  author    = {J. W. P. Hirschfeld},
  publisher = {Oxford University PressOxford},
  title     = {Projective Geometries over Finite Fields},
  year      = {1998},
  doi       = {10.1093/oso/9780198502951.001.0001},
}

@Article{Meagher2021,
  author    = {Karen Meagher and Andriaherimanana Sarobidy Razafimahatratra and Pablo Spiga},
  journal   = {J. Combin. Theory Ser. A},
  title     = {On triangles in derangement graphs},
  year      = {2021},
  issn      = {0097-3165},
  pages     = {105390},
  volume    = {{\bf 180}},
  doi       = {10.1016/j.jcta.2020.105390},
  publisher = {Elsevier BV},
}

@Article{Hujdurovic2022,
  author    = {Ademir Hujdurović and Klavdija Kutnar and Bojan Kuzma and Dragan Marušič and Štefko Miklavič and Marko Orel},
  journal   = {Finite Fields Appl.},
  title     = {On intersection density of transitive groups of degree a product of two odd primes},
  year      = {2022},
  issn      = {1071-5797},
  pages     = {101975},
  volume    = {{\bf 78}},
  doi       = {10.1016/j.ffa.2021.101975},
  publisher = {Elsevier BV},
}

@Article{Pantangi2025,
  author    = {Venkata Raghu Tej Pantangi},
  journal   = {European J. Combin.},
  title     = {All 3-transitive groups satisfy the strict-{E}rdős–{K}o–{R}ado property},
  year      = {2025},
  issn      = {0195-6698},
  pages     = {104057},
  volume    = {{\bf 124}},
  doi       = {10.1016/j.ejc.2024.104057},
  publisher = {Elsevier BV},
}

@manual{sagemath,
	Key      = {SageMath},
	Author   = {{The Sage Developers}},
	Title    = {{S}ageMath, the {S}age {M}athematics {S}oftware {S}ystem ({V}ersion 10.9)},
	note     = {\url{https://www.sagemath.org}},
	Year     = {2026}
}
	\bibliographystyle{abbrv}
	
\end{document}